\documentclass[10pt]{article}
\usepackage{amsmath, amssymb, amsthm}
\usepackage{appendix}
\usepackage{multicol} 
\usepackage{makeidx}
\usepackage{cite}
\usepackage{captdef}

\makeindex
\usepackage{paralist}
\usepackage{psfrag}
\usepackage{graphicx}
\usepackage{graphics}
\usepackage{color}

\newcommand{\R}{\mathbb{R}}

\newtheorem{theo}{Theorem}[section]
\newtheorem{coro}[theo]{Corollary}
\newtheorem{lemma}[theo]{Lemma}
\newtheorem{prop}[theo]{Proposition}
\theoremstyle{definition}

\DeclareMathOperator{\trace}{trace}
\DeclareMathOperator{\dist}{dist}

\newcommand{\cA}{{\mathcal A}}   
   
\newcommand{\cC}{{\mathcal C}}

\newcommand{\cR}{{\mathcal R}}

\newcommand{\cU}{{\mathcal U}}

\renewcommand{\phi}{\varphi}

\renewcommand{\epsilon}{\varepsilon}

\numberwithin{equation}{section} 
\renewcommand{\sectionmark}[1]{\markboth{\MakeUppercase{#1}}{}}
\renewcommand{\sectionmark}[1]{\markright{\MakeUppercase{#1}}{}}
\title{Asymptotic axial symmetry of solutions of parabolic equations in bounded radial domains}
\author{Alberto Salda\~{n}a\footnote{Institut f\"{u}r Mathematik, Johann Wolfgang Goethe-Universit\"{a}t Frankfurt, Robert-Mayer-Str. 10, D-60054 Frankfurt, saldana@math.uni-frankfurt.de.}  \ \ \  \& \ \ \ Tobias Weth\footnote{Institut f\"{u}r Mathematik, Johann Wolfgang Goethe-Universit\"{a}t Frankfurt, Robert-Mayer-Str. 10, D-60054 Frankfurt, weth@math.uni-frankfurt.de.} }
\date{}

\begin{document}
\maketitle

\begin{abstract}
We consider solutions of some nonlinear parabolic boundary value
problems in radial bounded domains whose initial profile satisfy a
reflection inequality with respect to a hyperplane containing the
origin. We show that, under rather general assumptions, these
solutions are asymptotically (in time) foliated Schwarz symmetric,
i.e., all elements in the associated omega limit set are axially
symmetric with respect to a common axis passing through the origin and
nonincreasing in the polar angle from this axis. In this form, the
result is new even for equilibria (i.e. solutions of the corresponding
elliptic problem) and time periodic solutions.
\end{abstract}



\section{Introduction}
Consider the reaction-diffusion problem 
\begin{equation}
\begin{aligned}
\label{(P-simple)}
 &u_t = \Delta u +f(t,|x|,u), && \qquad (x,t) \in (0,\infty)\times B,\\
    &u(x,t) = 0,  && \qquad (x,t)\in \partial B\times (0,\infty),\\
   &u(x,0) = u_0(x), && \qquad x\in B,
\end{aligned}
\end{equation}
where $B$ is a bounded radial domain in $\mathbb R^N,\ N\geq 2$, i.e.,
a ball or an annulus in $\R^n$ centered at zero. If
the nonlinearity $f$ is continuous and locally Lipschitz in
$u$ uniformly in $t$ and $|x|$, it follows from standard semigroup theory that, for every $u_0
\in C(\overline B)$, the corresponding local (in time) problem
admits a unique solution $u \in C(\overline B \times [0,T(u_0)))$ for some
time $T(u_0)>0$. Moreover, it has been studied extensively in
recent years under which assumptions on the nonlinearity $f$ and the initial condition $u_0$ this
unique solution exists globally in time and 
\begin{equation}
  \label{eq:1}
\text{the corresponding orbit
$\{u(\cdot,t)\::\: t >0\}$ is relatively compact in $C(\overline
B)$.}  
\end{equation}
 We refer the reader to
\cite{quittner-souplet,cazenave-haraux,polacik} and the references
therein, where many specific examples are discussed which give rise to
this behavior. It is then natural to investigate the qualitative
asymptotic behaviour of these global solutions. In this paper, we are
mainly inspired by work of Pol\'{a}\v{c}ik \cite{polacik} who studied
asymptotic symmetry of these solutions. In particular,
he proved that, if $B$ is a ball, $f$ is nonincreasing in $|x|$ and
if $f$ satisfies some rather
mild regularity assumptions, every nonnegative solution of
(\ref{(P-simple)}) satisfying (\ref{eq:1}) is asymptotically radially
symmetric. More precisely, every function $z$ belonging to the omega limit set 
\begin{equation*}
 \omega(u):=\{z\in C(\overline{B}) :
 \|u(\cdot,t_n)-z\|_{L^\infty(B)} \to 0 \text{ for some }
 t_n\to\infty\}
\end{equation*}
is radially symmetric and decreasing in the radial variable (see
\cite[Corollary 2.6]{polacik}). This result is proved via a parabolic
version of the moving plane method relying on
subtle estimates on solutions to linear parabolic equations.
We recall that the moving plane method has its roots in earlier work
of Alexandrov \cite{alex} and Serrin \cite{serrin} for geometric problems and has been elaborated in the
seminal paper of Gidas, Ni and Nirenberg \cite{GNN}
in order to prove symmetry results for solutions of elliptic nonlinear
boundary problems.\\
The motivation of the present paper is that, to our
knowledge, so far no asymptotic symmetry result is available for
sign-changing solutions of (\ref{(P-simple)}) and, if $B \subset \R^n$
is an annulus or $f$ is increasing in $|x|$, also for nonnegative
solutions. In fact, in these situations, equilibrium solutions of
(\ref{(P-simple)}) in the case where $f=f(|x|,u)$ does not depend on $t$
may already have a
very complicated shape. In particular,
 for suitable data, solutions with arbitrarily many isolated local
 maxima close to the boundary have been constructed, see \cite{clapp-delPino-musso,peng,pistoia-serra}. 
Therefore any type of symmetry result in this setting requires
additional assumptions on the initial profile $u_0$. In this paper we
assume a simple reflection inequality with respect to a
hyperplane. In order to state this assumption and our symmetry result, we need to introduce some
notation. Let $S=\{x\in\mathbb R^N: |x|=1\}$ be the unit sphere in
$\R^N$.  For a vector $e\in S$, 
we consider the hyperplane $H(e):=\{x\in \mathbb R^N: x\cdot e=0\}$
and the half domain $B(e):=\{x\in B: x\cdot e>0\}.$  We write also
$\sigma_e:B\to B$ to denote reflection with respect to $H(e),$
i.e. $\sigma_e(x):=x-2(x\cdot e)e$ for each $x\in B.$ We say that a
function $u\in C(B)$ is \textit{foliated Schwarz symmetric with
respect to some unit vector $p\in S$} if $u$ is axially symmetric with
respect to the axis $\mathbb R p$ and nonincreasing in the polar angle
$\theta:= \operatorname{arccos}(\frac{x}{|x|}\cdot p)\in [0,\pi].$ The name {\em foliated Schwarz symmetry} was introduced in
\cite{smets} by Smets and Willem, and it is also called ``codimension
one symmetry'' or ``cap symmetry'' by other authors.  
We refer the reader to the survey article \cite{wethsurvey} and the
references therein for a
broader discussion on symmetry properties of this type and its
relationship to reflection inequalities. Finally we set 
$$
I:= \overline{\{|x|\::\:x \in B\}}.
$$ 
The following is our main result for problem (\ref{(P-simple)}).
\begin{theo}
\label{maintheo-simple}
Suppose that 
\begin{itemize}
\item [(f1)] the nonlinearity $f: [0,\infty) \times I \times \R \to
\R$, $(t,r,u) \mapsto f(t,r,u)$ is continuous in $t,r$ and
locally Lipschitz in $u$ uniformly with respect to $t$ and $r$, i.e. for every $K>0$
there is $L=L(K)>0$ such that 
$$
|f(t,r,u_1)-f(t,r,u_2)| \le L |u_1-u_2|
$$
for all $(t,r) \in [0,\infty) \times I$ and $u_1,u_2 \in [-K,K]$.
\item[(f2)] $f(\cdot,\cdot,0)$ is bounded on $[0,\infty) \times I$.
\end{itemize}
Assume furthermore that $u\in C^{2,1}(B\times (0,\infty))\cap
C(\overline B\times[0,\infty))$ is a classical solution of
(\ref{(P-simple)}) such that  
\begin{itemize}
\item [(U1)] there is $e\in S$ such that $u_0\geq u_0\circ\sigma_e$ and $u_0\not\equiv u_0\circ\sigma_e$ in $B(e).$
\item [(U2)] $\|u(\cdot,t)\|_{L^\infty(B)}$ is uniformly bounded in $t.$
\end{itemize}
Then $u$ is asymptotically foliated Schwarz symmetric with
respect to some $p \in S$, i.e. all elements in $\omega(u)$ are
foliated Schwarz symmetric with respect to $p$.
\end{theo}

Indeed, we will prove a more general version of this result in Section~\ref{sec:framework} below, dealing with a more general class of evolution
problems similarly as in \cite{polacik}. An immediate corollary of
Theorem~\ref{maintheo-simple} is the following.

\begin{coro}
\label{maincoro-simple}
\begin{itemize}
\item[(i)] Let $f:I \times \R \to
\R$, $(r,u) \mapsto f(r,u)$ be continuous in $r \in I$ and
locally Lipschitz in $u$ uniformly with respect to $r$. Moreover, let $u \in
C(\overline B) \cap C^2(B)$ be a classical solution of the elliptic
problem  
\begin{equation}
\begin{aligned}
\label{(P-simple-elliptic)}
-\Delta u &= f(|x|,u), && \qquad \text{in $B$},\\
    u(x) &= 0,  && \qquad \text{on $\partial B$},
\end{aligned}
\end{equation}
such that (U1) holds for $u$ in place of $u_0$. Then $u$ is foliated
Schwarz symmetric with respect to some $p \in S$. 
\item[(ii)] Suppose that $f:[0,\infty) \times I \times \R \to
\R$ satisfies (f1) and is periodic in $t$, i.e. there is $T>0$
such that $f(t+T,r,u)=f(t,r,u)$ for all $t,r,u$. Suppose furthermore
that $u$ is a $T$-periodic solution of (\ref{(P-simple)}), i.e.,
$u(x,t+T)=u(x,t)$ for all $x\in B, t \in [0,\infty)$,
and such that (U1) holds. Then $u(\cdot,t)$ is foliated
Schwarz symmetric with respect to some $p \in S$ for all times $t \in [0,\infty)$. 
\end{itemize}
\end{coro}

Both parts of this corollary are new. Under additional spectral assumptions on
the solution, statements similar to part (i) have been derived 
in \cite{pacella,gladiali} as an intermediate step in the proof of
symmetry results for solutions of (\ref{(P-simple-elliptic)}) with low
Morse index. For time periodic solutions as considered in (ii), no
previous symmetry result seems
to be available in the present setting. We note that results on radial symmetry of nonnegative time periodic solutions
had been obtained by Dancer and Hess \cite{dancer.hess:94} in the setting where
$B$ is a ball in $\R^N$ and $f$ is nonincreasing in $|x|$. 

We note
that an easy example giving rise to a (sign changing) nonradial but
foliated Schwarz symmetric solution of (\ref{(P-simple-elliptic)})
-- and thus also of (\ref{(P-simple)}) -- is given by $f(|x|,u)=
\lambda_2 u$, where $\lambda_2$ is the second Dirichlet eigenvalue
of the Laplacian. It is known that every corresponding eigenfunction
is of the form $u(x)= j(|x|)\frac{x \cdot p}{|x|}$ for some $p \in S$
and some positive function $j$ on $I$, so $u$ is obviously nonradial but
foliated Schwarz symmetric with respect to $p$.

Our approach to prove (a more general version of)
Theorem~\ref{maintheo-simple} is by a rotating plane argument, which
should be seen as a variant of the moving plane method. The hypothesis
(U1) will allow us to start this method. In contrast to
the usual moving plane method on bounded domains in the form developed
in \cite{GNN} for elliptic and in \cite{polacik} for parabolic problems, the symmetry axis is not fixed a priori by assumption
(U1). Moreover, the rotating plane method alone only gives rise to
local monotonicity with respect to every (cylindrical) angle. An extra argument is needed to
translate this information into foliated Schwarz symmetry, see
Proposition~\ref{sec:char-foli-schw-1} below. Note also that assumption (U1)
does not imply that the functions in $\omega(u)$ are strictly
decreasing in the polar angle from the symmetry axis. For instance, in
case $B$ is a ball, $f$ is decreasing in $|x|$ and $u_0 \in C(\overline B)$ is a nonnegative function satisfying (U1), the
above-mentioned result \cite[Corollary 2.6]{polacik} of
Pol\'{a}\v{c}ik yields that $\omega(u)$ only consists of radial functions.   

In the elliptic setting, the rotating plane method was
used in combination with other arguments by Pacella and the
second author \cite{pacella} to prove -- under some convexity hypothesis on the
nonlinearity -- foliated Schwarz symmetry of
solutions with low Morse index. Later, this result was extended in \cite{gladiali} to
unbounded domains under additional restrictions. The rotating plane method in the elliptic setting
relies on different forms of the maximum principle (e.g., the maximum
principle for small domains, see \cite{bere}). In the parabolic
setting, the argument relies in a more subtle way on Harnack type
inequalities and related estimates for
linear equations. These estimates have been developed in a very useful
form by Pol\'{a}\v{c}ik in order to derive asymptotic symmetry results
in the Steiner symmetric setting \cite{polacik}, and we will make use
of them in the present framework.     

The paper is organized as follows. In Section~\ref{sec:framework} we present the general
framework for our symmetry results in the context of fully nonlinear
equations. In Section~\ref{sec:char-foli-schw} we provide a new
characterization of foliated Schwarz symmetry which is useful in
combination with the rotating plane method.  In
Section~\ref{sec:line-parab-probl}, we recall some estimates for linear
parabolic equations derived by Pol\'{a}\v{c}ik in \cite{polacik}, and
we introduce a family of linear parabolic problems associated with the
nonlinear problem. Finally, in Section
\ref{sec:rotating-hyperplanes} we apply the rotating plane method to the parabolic problem and prove the main result.

We add some closing remarks. Although Corollary~\ref{maincoro-simple} is an
immediate consequence of Theorem~\ref{maintheo-simple}, it can also be
derived independently by a somewhat simpler argument not relying on
the deep estimates in \cite{polacik}. In order to keep this paper
short, we leave the details to the reader.\\ 
In the present paper, we always consider a bounded radial domain. It
is natural to ask whether similar results are available in the case
where $B= \R^N$ or $B$ is the exterior of a ball in $\R^N$. This is
part of work in progress. We note
that, in a somewhat restricted setting, Pol\'{a}\v{c}ik \cite{polacik-unbounded} 
also developed a parabolic version of the moving plane method for the
case where the underlying domain is the entire space. However,
it is not straightforward to extend the parabolic {\em rotating plane
argument} to the unbounded setting, since additional obstacles
arise. In particular it seems more difficult than in
\cite{polacik-unbounded} to start the method and to analyze extremal hyperplanes, since the
local behaviour of $f$ close to zero cannot be used in the same way as
in \cite{polacik-unbounded}.

\section{The framework}
\label{sec:framework}
In this section we set up a more general framework for our symmetry
results. The setting is strongly inspired by \cite{polacik}. We
consider the fully nonlinear parabolic problem 
\begin{equation}
\label{(P)}
\begin{aligned}
 &u_t = F(t,x,u,\nabla u,D^2 u), && \qquad (x,t) \in (0,\infty)\times B,\\
    &u(x,t) = 0,  && \qquad (x,t)\in \partial B\times (0,\infty),\\
   &u(x,0) = u_0(x), && \qquad x\in B,
\end{aligned}
\end{equation}
where, as before, $B$ is a bounded radial domain in $\mathbb R^N,\
N\geq 2$, and $D^2 u=(u_{x_i x_j})_{i,j=1}^N \in \R^{N \times N}$ is
the Hessian of $u$. As for the right hand side of \eqref{(P)}, we consider the
following assumptions. 
\begin{itemize}
 \item [(F1)] \emph{Reflection invariance:} We have 
$$ 
F: [0,\infty) \times \overline B \times {\cal B} \to \mathbb R,
$$
where ${\cal B}$ is an open convex set in $\mathbb R\times \mathbb R^N\times
\mathbb R^{N\times N}$ such that $B \times {\cal B}$ is invariant under the
transformations
\begin{align*}
(x,u,p,q) \mapsto (Rx,u,R p,R q R),  \text{ for every hyperplane
  reflection $R \in \R^{N \times N}.$}
\end{align*}
Moreover,
$$
F(t,Rx,u,R p,R q R)=F(t,x,u,p,q)
$$
for every hyperplane reflection $R \in \R^{N \times N}$ and 
  $(t,x,u,p,q)\in (0,\infty) \times B \times {\cal B}$.
\item [(F2)] \emph{Regularity:} $F$ is continuous on $[0,\infty)\times \overline{B}\times {\cal B}$ and Lipschitz in $(u,p,q),$ uniformly with respect to $x$ and $t,$ i.e., there is $L>0$ such that
\begin{align*}
 \sup_{x\in B,t\geq 0}|F(t,x,u,p,q)-F(t,x,\tilde u,\tilde  p,\tilde  q)|\leq L |(u,p,q)-(\tilde u,\tilde  p,\tilde  q)|
\end{align*}
for all $(u,p,q)-(\tilde u,\tilde  p,\tilde  q)\in {\cal B}.$  Moreover, $F$ is differentiable with respect to $q$ on $[0,\infty)\times B\times {\cal B}.$
\item [(F3)] \emph{Boundedness:} $(0,0,0)\in {\cal B}$ and the function $F(\cdot,\cdot,0,0,0)$
  is bounded on $[0,\infty) \times B$.
\item [(F4)] \emph{Ellipticity:} There is a constant $\alpha_0>0$ such that
\begin{align*}
 \partial_{q_{ij}} F(t,x,u,p,q)\xi_i\xi_j \geq \alpha_0 |\xi|^2
\end{align*}
for all $(t,x,u,p,q)\in [0,\infty)\times B\times {\cal B}$ and $\xi\in
\mathbb R^N.$ Here and below, we use the summation convention (summation over repeated indices).
 \end{itemize}

We point out that these hypothesis are closely related to the ones in
\cite[Section 2]{polacik}. However, in contrast to \cite{polacik} we make no monotonicity assumptions on the nonlinearity and it may also include terms depending on the radial derivative of $u.$ So this allows us to also consider equations like
\begin{align*}
u_t=g(t,|x|,u,|\nabla u|,\Delta u)+d(|x|)\nabla u\cdot x,\ \ \ (x,t)\in B\times[0,\infty),
\end{align*} 
where $r \mapsto d(r)$ is a continuous function on $\mathbb R$,
$g=g(t,r,u,\eta,\xi)$ is continuous on $\mathbb R^5$ and Lipschitz in
$(u,\eta,\xi)$ uniformly in $(t,r)$, $g_\xi$ exist everywhere and
$g_\xi\geq \alpha_0$ for some positive constant $\alpha_0.$

The symmetry result which we want to prove
in this general setting relies also on assumptions $(U1)$ and $(U2)$ for a
fixed solution of (\ref{(P)}), which were stated in the
introduction. 

\begin{theo}\label{teodirichlet}
Assume $(F1)-(F4)$, and let $u\in C^{2,1}(B\times (0,\infty))\cap
C(\overline B\times[0,\infty))$ be a classical solution of the problem
$\eqref{(P)}$ satisfying assumptions (U1) and (U2). Then $u$ is asymptotically
foliated Schwarz symmetric with respect to some $p \in S$, 
i.e. all the elements in $\omega(u)$ are foliated Schwarz symmetric
with respect to $p$.
 \end{theo}

We quickly show how Theorem~\ref{teodirichlet} implies
Theorem~\ref{maintheo-simple}. Indeed, by (U2), we have
$$
K:= \sup_{t\ge 0,x \in B}|u(x,t)|<\infty.
$$
Hence we may consider (\ref{(P-simple)}) as
a special case of (\ref{(P)}) with ${\cal B}=(-K-1,K+1) \times \mathbb R^N\times
\mathbb R^{N^2}$ and  
$$ 
F: [0,\infty) \times \overline B \times {\cal B} \to \mathbb R,
\qquad F(t,x,u,p,q)= \trace(q)+f(t,|x|,u).
$$
With this definition, assumptions (F1) and (F4) are obviously
satisfied. Moreover, (F2) and (F3) follow from assumptions (f1) and
(f2) of Theorem~\ref{maintheo-simple}, respectively. Hence the
assumptions of Theorem~\ref{maintheo-simple} imply those of
Theorem~\ref{teodirichlet}, and therefore
Theorem~\ref{maintheo-simple} follows.

We point out that, as noted in \cite[Proposition
2.7]{polacik}, assumptions (F2),(F3) and (F4) on the nonlinearity $F$
ensure that, for every solution $u\in C^{2,1}(B\times (0,\infty))\cap
C(\overline B\times[0,\infty))$ of (\ref{(P)}) satisfying (U2),
\begin{align}\label{hoeldercontinuos}
\sup_{\genfrac{}{}{0pt}{}{\genfrac{}{}{0pt}{}{\scriptstyle{x,\bar x\in
        B, t,\bar t\in[s,s+1],}}{\scriptstyle{x \not= \bar x, t \not=
        \bar t,}}}{\scriptstyle{s\geq 1}}
} \frac{|u(x,t)-u(\bar x,\bar t)|}{|x-\bar x|^\alpha+|t-\bar
  t|^{\frac{\alpha}{2}}}< \infty \qquad \text{for some $\alpha>0$.}
\end{align}
Indeed, being a bounded radial domain, $B$ is smoothly bounded and
therefore satisfies assumption (A) of \cite[Proposition
2.7]{polacik}. It follows immediately from (\ref{hoeldercontinuos}) that the orbit
$\{u(\cdot,t)\::\: t >0\}$ is relatively compact in $C_0(\overline
B)$ and that $\omega(u)$ is a nonempty compact subset of $C_0(\overline
B)$ satisfying $\dist(u(t,\cdot),\omega(u)) \to 0$ in $C_0(\overline
B)$ as $t \to \infty$.

We finally note that hypothesis (F3) can be dropped if we assume
instead a priori that $\{u(\cdot,t):t>0 \}$ is an equicontinuous subset of $C(\overline{B}).$


\section{Characterizations of foliated Schwarz symmetry}
\label{sec:char-foli-schw}
As before, $B$ denotes a radial subdomain of $\R^N$,
$N\geq 2,$ and $I$, $S$, $H(e)$, $B(e)$ and $\sigma(e)$ are defined as in
the introduction for $e \in S$. We start by proving an auxiliary lemma.

\begin{lemma}\label{lemma1}
 Let $v\in C(\mathbb R)$ be an even and $2\pi$-periodic function, 
and let $\cR$ denote the points of reflectional symmetry of
$v$. If, for some $\eta \in \R$,  
 \begin{equation}
\label{geo4}     
   \begin{aligned}
  v(\eta + \phi)&\geq v(\eta -\phi)\ \  \ \text{ for all } \phi\in
  [0,\pi] \text{ and }\\
  v(\eta + \phi_0)&> v(\eta -\phi_0)\ \  \text{ for some }\phi_0\in(0,\pi).
   \end{aligned}
 \end{equation}
then we have $\cR= \{n \pi \::\: n \in \mathbb{Z}\}$. 
\end{lemma}

\begin{proof}
>From the fact that $v$ is even, continuous, $2\pi$-periodic and non-constant by assumption, it is
easy to deduce that $\cR= \{\frac{n \pi}{k}\,:\, n \in \mathbb{Z}\}$ with some positive
integer $k$. We suppose by contradiction that $k \ge 2$. Then $v$ is $\frac{2\pi}{k}-$periodic, and for a suitable
translation $w$ of $v$ we can assume that there is some $\eta\in(0,\frac{\pi}{k})$ and some $\phi_0\in (0,\frac{2\pi}{k}]$ such that
\begin{align*}
w(\phi)&= w(-\phi) &\text{ for all } \phi\in \mathbb R,\\
w\bigg(\pm\frac{\pi}{k} + \phi\bigg)&= w\bigg(\pm\frac{\pi}{k}-\phi\bigg) &\text{ for all } \phi\in \mathbb R,\\
w(\eta+\phi)&\geq w(\eta-\phi) &\text{ for all } \phi\in (0,\pi),\\
w(\eta+\phi_0)&>w(\eta-\phi_0).
\end{align*}
Since $0<\frac{2\pi}{k}-\phi_0< \pi$, it follows that
\begin{align*}
w(\eta+\phi_0)&>w(\eta-\phi_0)=w\bigg(-\frac{2\pi}{k}-(\eta-\phi_0)\bigg)=w\bigg(\eta+\frac{2\pi}{k}-\phi_0\bigg)\\
&\geq w\bigg(\eta-\frac{2\pi}{k}+\phi_0\bigg)=w\bigg(\frac{2\pi}{k}-\eta-\phi_0\bigg)=w(\eta+\phi_0),
\end{align*}
which yields a contradiction. Hence $k=1$, and thus the claim follows.
\end{proof}

Now we generalize a result of Brock (\cite{brock}, Lemma 4.2)
to characterize sets of foliated Schwarz symmetric functions with
respect to a common direction. 

\begin{prop}\label{newprop}
Let $\cU$ be a set of continuous functions defined on a radial domain $B\subset
 \mathbb R^N,$ $N\geq 2.$ Define
 \begin{equation}
   \label{eq:5}
M:=\{e\in S \mid u(x) \ge u(\sigma_e(x)) \text{ for all }x\in B(e) \text{ and } u \in \cU\}.
 \end{equation}
If
\begin{equation}\label{eq:3}
S= M \cup -M,
\end{equation}
i.e., if for all $e \in S$ we have
$$
u \geq u \circ \sigma_e\; \text{in $B(e)$ for all $u \in
  \cU$}\qquad \text{ or }\qquad u \leq u \circ \sigma_e \; \text{in
  $B(e)$ for all $u \in \cU$},
$$
then there is $p
 \in S$ such that every $u \in \cU$ is foliated Schwarz symmetric with
 respect to $p$.  
\end{prop}

\begin{proof}
We start by constructing orthogonal unit vectors $e_1,\dots,e_{N-1}$ such that 
\begin{equation}
  \label{eq:4}
u \equiv u \circ \sigma_{e_i} \qquad \text{for $i=1,\dots,N-1$ and every $u
  \in \cU.$}  
\end{equation}
For this we first consider the set 
\begin{align*}
\cA_1:= \{e \in S\::\: u(x)>u(\sigma_e (x))\; \text{for some $u \in
  \cU$ and some $x \in B(e)$}\}.
\end{align*}
By (\ref{eq:3}) we have $\cA_1 \subset M$, and $\cA_1$ does not
contain antipodal points. Moreover, $\cA_1$ is a relatively
open subset of $S$. If $\cA_1$ is empty, then $u \equiv u \circ \sigma_e$ for
any $u \in \cU$ and $e \in S$, so any choice of orthonormal
vectors $e_1,\dots,e_{N-1}$ satisfies (\ref{eq:4}). Hence we may assume that
$\cA_1 \not= \varnothing.$ Then also the relative boundary $\partial
\cA_1$ of $\cA_1$ in $S$ is non-empty. Let $e_1 \in \partial
\cA_1$; then any $u \in \cU$ satisfies $u \equiv u \circ \sigma_{e_1}$. Next we consider 
\begin{align*}
\cA_2:=\{e \in S \cap H(e_1)\::\: u(x) >u(\sigma_e (x)) \text{ for some
  $u \in \cU$ and some $x \in
  B(e)$}\}.
\end{align*}
If $\cA_2$ is empty, then may complement $e_1$ with any choice of orthonormal vectors
$e_2,\dots,e_{N-1}$ in $S \cap H(e_1)$ to obtain (\ref{eq:4}). If $\cA_2$
is nonempty, then -- by the same argument as above -- also the
relative boundary $\partial \cA_2$ of $\cA_2$ in $S \cap H(e_1)$ is nonempty, and every vector
$e_2 \in \partial \cA_2$ satisfies $u \equiv u \circ \sigma_{e_2}$ for every
$u \in \cU$. Successively we find orthogonal vectors
$e_1,\dots,e_{N-1} \in S$ such that (\ref{eq:4}) holds (then the process stops since $S \cap H(e_1)\cap
H(e_2) \cap \dots \cap H(e_{N-1})$ consists merely of two antipodal
points).\\
Without loss of generality, we may now assume that the vectors $e_1,\dots,e_{N-1}$
satisfying (\ref{eq:4}) are the first $N-1$ coordinate
vectors. Next we show that every hyperplane containing the $x_N$-axis
is a symmetry hyperplane for every $u \in \cU$. For this let $q=(q_1,\ldots,q_{N})
\in S$ be such that $\mathbb R e_N \subset H(q).$  By \eqref{eq:3} we
can assume that $q\in M$ (otherwise we replace $q$ by $-q$). Since
$q_N=0$, for $x\in B(q)$ we have that
$[\sigma_{e_1} \circ \ldots \circ \sigma_{e_{N-1}}](x)=-\sigma_{e_N}(x)\not\in B(q),$
and from (\ref{eq:4}) we deduce that
$$
u(x)=u(-\sigma_{e_N}(x))\leq
u(\sigma_q(-\sigma_{e_N}(x)))=u(-\sigma_{e_N}(\sigma_q(x)))=u(\sigma_q(x))\leq
u(x)
$$ 
for every $u \in \cU$. Hence $u\equiv u\circ\sigma_q$ for every $u \in \cU$, as
claimed. We conclude that every $u \in \cU$ is axially symmetric with respect to the
axis $\R e_N.$  

To complete the proof of foliated Schwarz symmetry, we may now
restrict to any two-dimensional subspace of $\R^N$ containing the axis
$\R e_N$, hence we may assume that $N=2$ from now on. Let $u\in\cU$ be a non radial function. Then there are ${e_*}\in S$ and $x\in B({e_*})$ such that ${e_*}\cdot e_2>0$ and
\begin{align}
 u(x)&>u(\sigma_{e_*}(x)) \text{ or } \label{case1}\\
u(x)&<u(\sigma_{e_*}(x)).\label{case2}
\end{align}

Assume \eqref{case1} first. Writing $u=u(r,\phi)$ in (permuted) polar
coordinates with $x_1= r \sin \phi$ and $x_2= r \cos \phi$, we get that $u$
is even in $\phi$, and that there are $r>0$ and $\eta \in (-\pi,0)$ such
that (\ref{geo4}) holds for the function $\R \to \R,\: \phi \mapsto u(r,\phi)$. Hence by Lemma~\ref{lemma1} there are no other
points of reflectional symmetry of this function in $(-\pi,0)$ except the origin, and
by \eqref{eq:3} this implies that for every $e\in S$ with $e\cdot
e_2>0$ we have $u\geq u\circ \sigma_e$ and $u \not \equiv u\circ
\sigma_e$ in $B(e)$. Then again by (\ref{eq:3}) we have that
\begin{align*}
u \ge u \circ \sigma_e \text{ in } B(e) \text{ for all } u \in \cU \text{ and all } e \in
  S \text{ with } e \cdot e_2 \ge 0,
\end{align*}
and this readily implies that every $u \in \cU$ is foliated
Schwarz symmetric with respect to the unit vector $e_2$.\\
A similar argument shows that, if we assume \eqref{case2} then every $u \in \cU$ is foliated
Schwarz symmetric with respect to the unit vector $-e_2$. The proof is finished.
\end{proof}

The following Proposition characterizes foliated Schwarz symmetry
by properties related to the method of rotating planes.

\begin{prop}
\label{sec:char-foli-schw-1}
 Let $\cU$ be a set of continuous functions defined on a radial domain $B\subset
 \mathbb R^N,$ $N\geq 2,$ and let $M$ be defined as in
 (\ref{eq:5}). Moreover, let $\tilde e \in M$.\\
If for all two dimensional subspaces $P\subseteq
 \mathbb R^N $ containing $\tilde e$ there are two different points
 $p_1, \ p_2$ in the same connected component of $M\cap P$ such that
 $u \equiv u \circ \sigma_{p_1}$ and $u \equiv u \circ
 \sigma_{p_1}$ for every $u \in \cU$, then there is $p
 \in S$ such that every $u \in \cU$ is foliated Schwarz symmetric with
 respect to $p$.
\end{prop}

\begin{proof}
Let $P$ be a two dimensional subspace with
$\tilde e\in P$. By hypothesis there is some
connected component $K_P$ of $M\cap P$ and $p_1, p_2 \in K_P$ such
that $u \equiv u\circ \sigma_{p_1}$ and $u\equiv u\circ\sigma_{p_2}$
for every $u \in \cU$. We first show that 
\begin{equation}
  \label{eq:2}
\text{$K_P$ contains a closed halfcircle,}  
\end{equation}
i.e., $\{e\in S \cap P \::\: e\cdot e' \geq 0 \}\subseteq K_P$ for
 some $e'\in S$. We assume without loss of generality that 
$$
p_1=(1,0,\ldots,0),\qquad p_2=(\cos\psi,\sin\psi,0,\ldots,0)\; \  \text{ for some $\psi\in(0,2\pi]$}
$$
and 
$$
(\cos \phi,\sin\phi,0,\ldots,0)=:p_\phi\in M\qquad \text{for all } \phi\in[0,\psi]
$$
(because $p_1$ and $p_2$ are in the same connected component of $M
\cap P$). Let $u \in \cU$. Using polar coordinates, we define 
\begin{align*}
\tilde v(r,\phi,x') := u(r\cos\phi,r \sin \phi, x')=u(x) 
\end{align*}
with $x\in B,\ x'=(x_3,\ldots,x_{N}) \in \R^{N-2},\ \phi\in
\mathbb R,$ and $r=|x| \in I$. If, independently of
the choice of $u \in \cU,$ $\tilde v$ does not depend on $\phi$, then
$M \cap P= S \cap P$ and so
(\ref{eq:2}) holds trivially. So, we may suppose that
$u \in \cU$ was chosen such that the function
$$
v: \R \to \R,\qquad v(\phi):=\tilde v (r_,\phi,x')
$$
is non-constant for some fixed $r>0$ and $x' \in \R^{N-2}$. By assumption, we then have
\begin{align}
 v(\phi)&=v(-\phi), \hspace{1cm}  \phi\in \mathbb R,\nonumber\\
 v(\psi+\phi)&=v(\psi-\phi), \hspace{.45cm}\phi\in \mathbb R,\nonumber\\
 v(\eta+\phi)&\geq v(\eta-\phi), \hspace{.5cm} \eta \in (0,\psi),\phi\in (0, \pi),\label{eq0n}
 \end{align}
i.e. $v$ has two points of reflectional symmetry, one at zero, and one
at $\psi,$ and the points in between satisfy the defining property
of $M.$  Since the function is non-constant, the inequality in
\eqref{eq0n} must be strict for some $\eta \in (0,\psi)$ and $\phi \in
(0,\pi)$. Then, by Lemma \ref{lemma1}, we get that $u \not \equiv u
\circ \sigma_{p_\phi}$ for $\phi \in (0,\pi)$. By assumption, we then
conclude that $p_2 \not= p_\phi$ for $\phi
\in (0,\pi)$, and therefore $\psi \ge \pi$. Hence (\ref{eq:2})
holds, as claimed.\\
Now since (\ref{eq:2}) holds independently of $P$, we conclude that,
for all $e \in S$ we have $e \in M$ or $-e \in M$, so that
(\ref{eq:3}) holds. Hence the assertion follows from Proposition~\ref{newprop}.
\end{proof}

\section{Linear parabolic problems associated with reflections at hyperplanes}
\label{sec:line-parab-probl}

To use the rotating plane method in the parabolic setting, the crucial
step is to consider the linear problem satisfied by the
difference between a solution of \eqref{(P)} and its reflection at a
hyperplane. In order to deal with this problem, we first quote
estimates derived by Pol\'{a}\v{c}ik \cite{polacik} for linear
parabolic equations in a general setting. So in the following, we
consider the general linear equation
\begin{align}
 v_t &= a_{ij}(x,t) v_{x_ix_j}+b_i(x,t)v_{x_i} + c(x,t)v, \ \ (x,t)\in U\times(\tau,T),\label{linear1}\\
 v &= 0, \ \ (x,t)\in\partial U\times(\tau,T),\label{linear2}
\end{align}
where $U$ is an open subset of some fixed bounded domain $\Omega
\subset \R^N$, $0\leq \tau < T \leq \infty,$ the coefficients $a_{ij},b_i,c$ are defined on $U\times(\tau,T),$  are measurable and for some positive constants $\alpha_0,$ $\beta_0$ satisfy that
\begin{align}
\begin{split}\label{coeff}
|a_{ij}(x,t)|,|b_{i}(x,t)|,|c(x,t)|&<\beta_0, \hspace{.9cm} x\in U,t\in[\tau,T),i,j=1,\ldots,N,\\
a_{ij}(x,t)\xi_i\xi_j&\geq \alpha_0 |\xi|^2,\hspace{.3cm} x\in U,t\in[\tau, T),\xi\in\mathbb R^N.
\end{split}
\end{align}

When referring to a solution of equation \eqref{linear1}, we mean a
function $v$ in the Sobolev space $W^{2,1}_{N+1,loc}(U\times
(\tau,T))$ such that \eqref{linear1} is satisfied almost everywhere. 
A solution of the boundary value problem
\eqref{linear1},\eqref{linear2} is in addition supposed to be
continuous on $\overline{U}\times[\tau,T)$ and to satisfy
\eqref{linear2} in the pointwise sense. The following two results are special cases of Theorems by
Pol\'{a}\v{c}ik (see \cite[Lemma 3.4 and Theorem 3.7]{polacik}).

\begin{lemma}{ (Special case of \cite[Lemma 3.4]{polacik})}\\
\label{lemapol} Let $\Omega$ be a bounded domain. Given $d,\theta>0,$ there is a positive constant $\kappa$ determined only by $N, \operatorname{diam}(\Omega),\alpha_0, \beta_0,d$ and $\theta$ with the following property.  If $D,U$ are domains in $\Omega$ with $D\subset \subset U,$ $\operatorname{dist}(\overline{D}, \partial U)\geq d,$ and $v\in C(\overline{U}\times[\tau,\tau+4\theta])$ is a solution of \eqref{linear1},~(\ref{linear2}) on $U\times (\tau,\tau+4\theta),$ then 
\begin{align*}
 \inf_{D\times (\tau+3\theta,\tau+4\theta)}v \geq \kappa \sup_{D\times(\tau+\theta,\tau+2\theta)}v-e^{4m\theta}\sup_{\partial_P(U\times(\tau,\tau+4\theta))}v^-,
\end{align*}
where $m=\sup\limits_{U\times (\tau,\tau+4\theta)}c.$  
\end{lemma}
Here $v^+:= \max\{v,0\}$ and $v^-:=-\min \{v,0\}$
denote the usual positive and negative parts of a function
$v$. Moreover, $\partial_P(U\times(\tau,\tau+4\theta))= \overline U
\times \{\tau\} 
\cup \partial U \times (\tau,T)$ denotes the parabolic boundary of $U\times(\tau,\tau+4\theta)$. In the following, we also use the notation 
$$
 \operatorname{inrad}(\Omega):= \sup\{r>0: B_r(x)\subset \Omega
 \text{ for some } x\in \Omega\},
$$
where $B(x,r)=B_r(x)=\{y\in \mathbb R^N : |x-y|<r\}$. 

\begin{theo}(Special case of \cite[Theorem 3.7]{polacik})\\
\label{polacikthm}
Fix $\rho\in(0,\frac{\operatorname{diam}(\Omega)}{2})$.Then there
is
$$
\delta=\delta(N,\operatorname{diam}(\Omega),\alpha_0,\beta_0,\rho)>0
$$
and, for every $d,\theta>0$,
$$
\mu=\mu(N,\operatorname{diam}(\Omega),\alpha_0,\beta_0,d,\theta,\rho)\in(0,1]  
$$
with the following properties: If $D\subset U$ are subdomains of $\Omega$ satisfying 
\begin{align*}
 \operatorname{inrad}&(D)>\rho, \ \ |U\backslash \overline{D}|<\delta,\\
&\operatorname{dist}(\overline{D},\partial U)>d,
\end{align*}
if $v\in C(\overline{U}\times[\tau,\infty))$ is a solution of a problem \eqref{linear1},\eqref{linear2} whose coefficients satisfy \eqref{coeff} (with $T=\infty$), and if
\begin{align}
  &v(x,t) >0 \quad \text{for $(x,t)\in \overline{D}\times [\tau,\tau + 8\theta)$,}\nonumber\\
 &\|v^-(\cdot,\tau)\|_{L^\infty(U\backslash D)} \leq \mu \|v\|_{L^{\infty}(D\times(\tau+\theta,\tau+2\theta)}, \label{p}
\end{align}
then the following statements hold true:
\begin{itemize}
 \item [(S1)] $v(x,t)>0$ for all $(x,t)\in \overline{D}\times [\tau,\infty).$
 \item [(S2)] $\|v^-(x,t)\|_{L^\infty(U)}\to 0$, as $t\to\infty.$
\end{itemize}
\end{theo}
 
We now come back to the linear problem satisfied by the
difference between a solution of \eqref{(P)} and its reflection at a
hyperplane. So as before, $B$ denotes a radial subdomain of $\R^N$,
$N\geq 2,$ and $I$, $S$, $H(e)$, $B(e)$ and $\sigma(e)$ are defined as in
the introduction for $e\in S$. Moreover, we let $u$ denote a solution of
\eqref{(P)}, and for $e\in S$ we define 
$$
w_e(x,t):=u(x,t)-u(\sigma_e(x),t)\qquad \text{for $(x,t)\in B(e)\times [0,\infty).$}
$$
Then $w_e$ is a solution of the problem
\begin{equation}\label{Pe}
 \begin{aligned}
 &\partial_t w_e=a^e_{ij}(x,t)(w_e)_{x_i x_j}+b^e_i(x,t)(w_e)_{x_i}+c^e(x,t)w_e, && (x,t)\in B(e)\times (0,\infty),\\
 &w_e(x,t) = 0, && (x,t)\in \partial B(e)\times (0,\infty),\\
 &w_e(x,0) = u_0(x)-u_0(\sigma_e(x)), && x\in B(e),\\
  \end{aligned}
\end{equation}
where the coefficients are obtained,
as in \cite{polacik}, via the Hadamard formula. To make this precise,
let $u^e(x,t):=u(\sigma_e(x),t)$ and consider
\begin{align*}
c^e(x,t)&:=
\begin{cases}
\int_0^1 F_u(t,|x|,su+(1-s)u^e,D u,D^2 u))ds, &\text{ if $u(x,t) \neq u^e(x,t),$ }\\
0, &\text{ if $u(x,t) = u^e(x,t),$ }
\end{cases}\\
b^e_i(x,t)&:=
\begin{cases}
\int_0^1 F_{p_i}(t,|x|,u^e, \ldots, u^e_{x_{i-1}}, su_{x_{i}}\\
\hspace{1cm}+(1-s)u_{x_{i}}^e,u_{x_{i+1}},\ldots,D^2 u))ds, &\text{ if $u_{x_{i}}(x,t) \neq u_{x_{i}}^e(x,t),$ }\\
0, &\text{ if $u_{x_{i}}(x,t) = u_{x_{i}}^e(x,t),$ }
\end{cases}\\
a_{ij}^e(x,t)&:=\int_0^1 F_{q_{ij}}(t,|x|,u^e,Du^e, \ldots, u^e_{x_{i^-}x_{j^-}}, su_{x_{i}x_j}\\
&\hspace{2cm}+(1-s)u_{x_{i}x_j}^e,u_{x_{i^+}x_{j^+}},\ldots,u_{x_N x_N}))ds,
\end{align*}
where $(i^-,j^-),(i^+,j^+)$ stand for the pairs of indices preceding,
respectively, following, $(i,j)$ within a fixed identification of $\R^{N\times N}$
 with $\mathbb R^{N^2}.$ 

 By (F1) and (F2) the integrals make sense
and give the right quotients for the right hand side of \eqref{Pe} to
be equal to the difference of $F(t,|x|,u,Du,D^2u)$ and
$F(t,|x|,u^e,Du^e,D^2u^e).$  

For every $z \in \omega(u)$, let
\begin{align*}
z_e \in C_0(\overline {B(e)}),\qquad z_e(x):= z(x)- z(\sigma_e (x)) \text{ for } x\in B(e).
\end{align*}
 Finally we define the set 
\begin{equation}
{\cal M}:=\{e\in S\mid z_e(x)\geq 0 \text{ for all } x\in B(e) \text{ and } z\in \omega(u) \}.\label{Mdef}   
\end{equation}
We remark that, as a consequence of (F2) and (F4), there is $\beta_0>0$ such that 
\begin{equation}
\label{beta0}
|c^e(x,t)|,|b^e_i(x,t)|,|a_{ij}^e(x,t)|<\beta_0 \qquad \text{and}\qquad a_{ij}(x,t)\xi_i\xi_j \geq \alpha_0 |\xi|^2
\end{equation}
for all $(x,t)\in B(e)\times
[0,\infty), i,j=1,\dots,N,\xi \in \R^N$ and $e\in S$ with $\alpha_0>0$ as in (F4). 

\section{Proof of the main result}
\label{sec:rotating-hyperplanes}
As before, $B$ denotes a radial subdomain of $\R^N$,
$N\geq 2,$ and $I$, $S$, $H(e)$, $B(e)$ and $\sigma(e)$ are defined as in
the introduction for $e\in S$. Moreover, for a fixed solution $u$ of
(\ref{(P)}) satisfying the assumptions of Theorem~\ref{teodirichlet},
we will make use of the definitions introduced in
Section~\ref{sec:line-parab-probl}. Recall, in particular,  the definition of $\cal M$ in \eqref{Mdef}.

\begin{lemma}\label{teopohob}
 Let $e\in S$ be as in (U1). Then there is some $\varepsilon>0$ such that $e'\in {\cal M}$ for all $e'\in S$ with $|e'-e|<\varepsilon.$
\end{lemma}
\begin{proof} If $e\in S$ is as in (U1), then it follows from
  (\ref{Pe}) and the parabolic strong maximum principle (see for example \cite{protter}) that
\begin{equation}\label{star}
w_e(x,t)>0 \text{ in } B(e)\times(0,\infty),
\end{equation}
and therefore $e\in{\cal M}.$  Let $\delta>0$ be 
chosen as in Theorem \ref{polacikthm} corresponding to $\Omega=B$,
$\rho:=\frac{\operatorname{inrad}(B)}{4}$, and $\alpha_0$, $\beta_0$ as in
(\ref{beta0}). Moreover, let $D\subset\subset B(e)$ be a subdomain such
that 
$|B(e)\backslash D|<\delta$ and
$\operatorname{inrad}(D)>\rho$. Put
$d:=\frac{\operatorname{dist}(\overline{D},\partial B(e))}{2}$,
$\theta:=1$, and let $\mu \in (0,1]$ be as in Theorem \ref{polacikthm}
(corresponding to these choices of $\Omega,\alpha_0,\beta_0,d,\theta$
and $\rho$). By \eqref{star} there exists some $\eta>0$ such that
$$
 w_e(x,t)>\eta>0, \ \ (x,t)\in \overline{D}\times [1, 9].
$$
Moreover, there is some
$\varepsilon>0$ such that for all  
$e'\in S$ with $|e-e'|<\varepsilon$ we have 
$$
D\subset\subset B(e'),\qquad |B(e')\backslash D|<\delta, \qquad
\operatorname{dist}(\overline{D},\partial B(e'))>d  
$$
and, as a consequence of continuity and \eqref{star}, 
\begin{align*}
&w_{e'}(x,t)>\frac{\eta}{2}>0, \ \ (x,t)\in \overline{D}\times [1, 9],\\
&\|w_{e'}^-(\cdot,1)\|_{L^\infty(B(e'))}\leq \frac{\mu\eta}{2}\leq \mu \|w_{e'}\|_{L^\infty(D\times [2,3])}.
\end{align*}
Hence for these $e' \in S$ the hypothesis of Theorem \ref{polacikthm}
are satisfied with $U=B(e')$, $\tau=1$, $\theta=1$ and $D$ as above, and thus we get that
 \begin{align*}
\|w_{e'}^-(\cdot,t)\|_{L^\infty(B(e'))}\to 0, \text{ as } t\to\infty.
\end{align*}
This shows $e' \in {\cal M}$ for $e' \in S$ with $|e-e'|<\varepsilon$, as claimed.
\end{proof}

\begin{lemma}\label{lobolemab}
Let $e\in {\cal M}.$ If there is some $\tilde z\in\omega(u)$ such that $\tilde z_e\not\equiv 0,$ then there is some $\varepsilon>0$ such that $e'\in{\cal M}$ for all $e'\in S$ with $|e-e'|<\varepsilon.$
\end{lemma}
\begin{proof} Since $\tilde z_e\not\equiv 0$ there is some
  $\alpha>0$ and $x_0\in B(e)$ such that $\tilde z_e (x_0)\geq
  2\alpha>0.$  Let $\delta>0$ be 
chosen as in Theorem \ref{polacikthm} corresponding to $\Omega=B$,
$\rho:=\frac{\operatorname{inrad}(B)}{4}$, and $\alpha_0$, $\beta_0$ as in
(\ref{beta0}). Moreover, let $D\subset\subset B(e)$ be a subdomain such
that $|B(e)\backslash D|<\delta$,
$\operatorname{inrad}(D)>\rho$ and $x_0 \in D$. Put
$d:=\frac{\operatorname{dist}(\overline{D},\partial B(e))}{2}$,
$\theta:=\frac{1}{8}$, and let $\mu \in (0,1]$ be as in Theorem \ref{polacikthm}
(corresponding to these choices of $\Omega,\alpha_0,\beta_0,d,\theta$
and $\rho$). Since $z_e\geq 0$ in $B(e)$ for all $z\in\omega(u)$ and, as remarked at
the end of Section~\ref{sec:framework}, $\dist(u(\cdot,t),\omega(u))
\to 0$ in $C_0(\overline B)$ as $t \to \infty$, there is some $T_0>0$ such that
\begin{equation}
  \label{eq:6}
 \| w^-_{e}(\cdot,t)\|_{L^\infty(B(e))}<\frac{\mu\kappa
   \alpha}{8}e^{-4\beta_0} \qquad \text{for $t\ge T_0,$}
\end{equation}
where $\kappa>0$ is the constant given by Lemma \ref{lemapol} for
$\Omega,\alpha_0,\beta_0,d$ as above and $\theta=1$. Next, we may pick $T_1 \ge T_0+1$ such that $\|w_e(\cdot,T_1)-\tilde
z_e\|_{L^\infty(B(e))}<\alpha$ and therefore $w_e(x_0,T_1)>\alpha$. We then apply
Lemma \ref{lemapol} to
$U=B(e)$, $\tau:=T_1+2$ and $\theta=1$ in order to get
\begin{align*}
 \inf_{D\times (\tau,\tau+1)}  w_e&\geq \kappa \|
 w_e^+\|_{L^\infty(D\times(\tau-2,\tau-1))}-e^{4\beta_0} \sup_{\partial_P(B(e)
   \times(\tau-3,\tau+1))}  w_e^-\\
&\geq \kappa\alpha-\frac{\mu\kappa \alpha}{8}\geq\frac{\kappa \alpha}{2}=:\eta>0.
\end{align*}
Moreover, there is some
$\varepsilon>0$ such that for all  
$e'\in S$ with $|e-e'|<\varepsilon$ we have 
$$
D\subset\subset B(e'),\qquad |B(e')\backslash D|<\delta, \qquad
\operatorname{dist}(\overline{D},\partial B(e'))>d  
$$
and, by continuity,
$$
\inf_{D\times (\tau,\tau+1)}  w_{e'} \ge \frac{\eta}{2}\qquad
\text{and}\qquad 
\|w^-_{e'}(\cdot,\tau)\|_{L^\infty(B(e'))}\le 
\|w^-_{e}(\cdot,\tau)\|_{L^\infty(B(e))}+\frac{\eta \mu}{4}.  
$$
Combining this with (\ref{eq:6}), we find that 
$$
\|w^-_{e'}(\cdot,\tau)\|_{L^\infty(B(e'))} \le
\frac{\eta\mu}{4}+\frac{\mu\kappa \alpha}{8}e^{-4\beta_0 } \le
\frac{\mu \eta}{2} \leq \mu \|  w^+_{e'}\|_{L^\infty({D\times (\tau+\frac{1}{8},\tau+\frac{1}{4})})}
$$
for every $e' \in S$ with $|e-e'|<\varepsilon$. In particular, for
these $e' \in S$ the hypothesis of Theorem \ref{polacikthm} are
satisfied with $U=B(e')$ and $\theta=\frac{1}{8}$, and therefore  
$$
 \|  w^-_{e'}(x,t)\|_{L^\infty(B(e'))}\to 0 \quad \text{as $t\to\infty$.}
$$
This yields $e'\in{\cal M}$ for all $e'\in S$ with $|e-e'|<\varepsilon,$
as claimed.
\end{proof}

We are now ready to prove the main symmetry result.

\begin{proof}[Proof of Theorem \ref{teodirichlet}] 
 
 Let $e\in S$ be as in (U1). Then, by Lemma \ref{teopohob}, there is $\varepsilon>0$ such that 
 \begin{equation}
e'\in {\cal M} \qquad \text{for all $e' \in S$ with $|e'-e|<\varepsilon.$} \label{Mopen}   
 \end{equation}
 Let $P$ be any two dimensional subspace of $\R^N$ containing $e$. Without loss of generality, we may assume that $e=(1,0,\ldots,0)$ and $ P = \{ x=(x_1,0,...,0,x_N) \mid x_1 , x_N \in \mathbb R \}.$  Define 
$$
e_\theta:=(\cos(\theta),0,...,0,\sin(\theta)), \qquad  z_{\theta}:=
z_{e_\theta} \in \cC_0(B(e_\theta))
$$
for $\theta \in \R$ and 
\begin{align*}
\Theta_1&:=\sup\{\theta>0 \::\: e_\phi\in {\cal M} \text{ for all } 0\leq \phi\leq \theta \},\\
 \Theta_2&:=\inf\{\theta<0 \::\: e_\phi\in {\cal M} \text{ for all } \theta \leq \phi\leq 0 \}.
   \end{align*}
We note that $\Theta_2<0<\Theta_1$ by
\eqref{Mopen}. If $\Theta_1-\Theta_2 \ge 2\pi$ (and in particular if
$\Theta_1= \infty$ or $\Theta_2=-\infty$), it immediately follows from the definition of ${\cal M}$ that every $H(e_\theta)$, $\theta \in \R,$ is a symmetry hyperplane for
all the elements in $\omega(u)$.  If both $\Theta_1$ and $\Theta_2$
are finite and $\Theta_1-\Theta_2<2\pi$, we have $z_{\Theta_1}\equiv z_{\Theta_2}\equiv 0$ for all
$z\in\omega(u)$ as a consequence of Lemma~\ref{lobolemab}, so that
$H(e_{\Theta_1})$ and $H(e_{\Theta_2})$ are symmetry hyperplanes for
all the elements in $\omega(u).$ Moreover, $e_{\Theta_1}\neq e_{\Theta_2}$ and
$e_\varphi\in {\cal M}$ for all $\varphi\in(\Theta_2,\Theta_1)$. Since
this can be done for all two dimensional subspaces $P$ of $\R^N$
containing $e$, we can use Proposition ~\ref{sec:char-foli-schw-1}, applied to ${\cU}= \omega(u)$, to
obtain the existence of $p \in S$ such that every $z \in \omega(u)$ is
foliated Schwarz symmetric with respect to $p$, as claimed.
\end{proof}

\section{Acknowledgements}
This work is part of a current Ph.D project supported by a joint grant
from CONACyT (Consejo Nacional de Ciencia y Tecnolog\'{\i}a - Mexico)
and DAAD (Deutscher Akademischer Austausch Dienst - Germany). The
authors would like to thank Nils Ackermann for helpful comments and the referee for his/her valuable
remarks and the suggestion to use Lemma~\ref{lemma1}.

\renewcommand{\sectionmark}[1]{\markboth{\MakeUppercase{#1}}{}}
\renewcommand{\sectionmark}[1]{\markright{\MakeUppercase{#1}}{}}

\end{document}